\documentclass[12pt]{amsart}
\usepackage{geometry,amssymb} 
\theoremstyle{plain}
\newtheorem{theorem}{Theorem}

\title{Quasi-triangle inequality for absolute correlation distance}
\author{Stanislav Dubrovskiy}
\email{duuubr@gmail.com}

\begin{document}

\maketitle

\begin{abstract}
We show that absolute correlation distance satisfies a $K$-relaxed triangle inequality, with the best $K=2$. 
\end{abstract}

\thispagestyle{empty}
High correlation is not a transitive relation between datasets (considered as vectors in $\mathbb{R}^n$): $u$ correlated with $v$ and $w$, does not in general imply $v$ and $w$ are correlated . For distances determined by a metric, there is the triangle inequality to quantify this. Correlation too gives rise to a distance - the absolute correlation distance, $d_{|corr|}$, which does not satisfy the triangle inequality (cf. \cite{DL}). See T. Tao's blog \cite{T} for discussion, examples and results. We show that it nevertheless satisfies a $K$-relaxed triangle inequality, with $K=2$ as the best constant (that is any $K<2$ fails).

Absolute correlation distance between two samples $X$ and $Y$ in $\mathbb{R}^n$ is given as 
\begin{equation}
d_{|corr|}(X,Y)=1-\left|\frac{\langle X-\bar{X},\, Y-\bar{Y} \rangle}{ \lVert X-\bar{X} \rVert \cdot \lVert Y-\bar{Y} \rVert } \right| .
\end{equation}
Here $\bar{X}$ is the arithmetic mean of $X$, and $ \lVert X \rVert$ - its euclidean norm.
While useful as evidence of linear relation between $X$ and $Y$, $d_{|corr|}$ is not a metric, in particular it does not satisfy triangle inequality. Take for example three vectors (of zero mean) forming angles of $\pi/2$, $\pi/4$, and $\pi/4$ (possible in dimensions $n\geq 3$): 
\begin{equation}
 \label{exmpl} 
1-|\cos\pi/2| \nleq 1-|\cos\pi/4| + 1- |\cos\pi/4|\ .
\end{equation}
In lieu of triangle inequality, an estimate of the type
\begin{equation}
\label{qT}
d_{|corr|}(X,Z)\leq K(d_{|corr|}(X,Y)+d_{|corr|}(Y,Z)),
\end{equation} 
for some $K$ (a $K$-relaxed triangle inequality), if true, would be useful as a measure of transitivity of correlation closeness. 

The above example (\ref{exmpl}) implies $K\geq \frac{1}{2-\sqrt 2}$, which is not far from the best possible $K=2$.
\begin{theorem}[Quasi-triangle inequality]
\label{qTLemma}
Absolute correlation distance $d_{|corr|}$ satisfies a $K$-relaxed triangle inequality, with the best $K=2$. Namely, for $X, Y, Z \in \mathbb{R}^n$:
\begin{equation}
\label{qTL}
d_{|corr|}(X,Z)\leq K(d_{|corr|}(X,Y)+d_{|corr|}(Y,Z))
\end{equation} 
for $K \geq 2$. The inequality fails (at least for some $X, Y, Z \in \mathbb{R}^n$) for $K<2$.
\end{theorem}
\begin{proof}
Without loss of generality, we can assume $X, Y, Z \in \mathbb{S}^{n-1}$ (a unit sphere in $\mathbb{R}^n$). 
In fact we can assume that $X, Y, Z \in \mathbb{S}^{n-2}=\mathbb{S}^{n-1} \cap  \Pi_{0}$, where $\Pi_{0}$ is a hyperplane of zero mean vectors, 
defined by its normal: ${\bf n}=(1, 1, \ldots, 1)$.
Then $d_{|corr|}(X,Y) =  1-|\cos\widehat{XY} |$, and (\ref{qTL}) can be rewritten as:
$$1-|\cos\widehat{XZ}\, | \leq K(1-|\cos\widehat{XY} |+1-|\cos\widehat{YZ}|)\ .$$
To simplify notation let us use $\gamma=\widehat{XZ}$, $\alpha=\widehat{XY}$ and $\beta=\widehat{YZ}$.
Considering angles (in radian measure) as their corresponding big circle arcs on the unit sphere, we can make use of the triangle inequality from Riemann geometry on $\mathbb{S}^{n-2}$: 
\begin{equation}
\label{tiRG} 
|\alpha-\beta|\leq\gamma\leq \alpha+\beta. 
\end{equation}

Call the left side of (\ref{qTL}) $f(\gamma)$, and its right hand side $g(\alpha,\beta)$:
\begin{equation*}
f(\gamma)=1-|\cos\gamma|,\quad g(\alpha,\beta)=K(1-|\cos\alpha|+1-|\cos\beta|).
\end{equation*}
First assume $\gamma= \alpha+\beta$. Since $0 \leq \alpha, \beta \leq \pi$ it would suffice to establish the inequality on $I=[0,\pi] \times [0,\pi]$. 
Cover $I$ with four regions: 
$$R_1=[0,\pi/2] \times [0,\pi/2],  R_2=[\pi/2,\pi] \times [0,\pi/2], R_3=[0,\pi/2] \times [\pi/2,\pi],$$ 
$\text{ and } R_4=[\pi/2,\pi] \times [\pi/2,\pi].$
It is enough to show 
\begin{equation}
\label{fqTi} 
f(\alpha+\beta) \leq g(\alpha,\beta)
\end{equation}
 in each of these regions. 

Begin with $R_1$. Since the left side of (\ref{fqTi}) has lines $\alpha+\beta=C$ (constant) as level curves, all we need to show is that minimum of $g$, restricted to each $f$-level curve, is not less than the value of $f$ on it:
\begin{equation}
\label{cts} 
f(C) \leq \min_{\alpha+\beta=C}g(\alpha,\beta)
\end{equation}
If $\beta=C-\alpha$ (on the level line), we see that $$g'_\alpha=K(\sin\alpha-\sin\beta), $$ and the minimum is achieved when $\alpha=\beta$. Thus we need to compare 
\begin{equation*}G(\alpha)=\left.g\right|_{\alpha=\beta}=2K(1-\cos\alpha)
\end{equation*}
with 
\begin{equation*}
F(\alpha)=\left.f\right|_{\alpha=\beta}=
\left[
\begin{array}{ll}
1-\cos{2\alpha}=2(1-\cos{\alpha})(1+\cos{\alpha}) & \text{if } 0 \leq \alpha \leq \pi/4,\\ 
1+\cos{2\alpha}\leq 1 & \text{if } \pi/4 \leq \alpha\leq \pi/2.\ 
\end{array}
\right.
\end{equation*} 
While $\alpha \leq \pi/4,\ (1+\cos{\alpha})\leq K$, and the inequality holds. At $\pi/4$ we check that 
\begin{equation}
\label{midway}
G(\pi/4)=K(2-\sqrt{2})>1,
\end{equation}
and increasing from there, while $F(\alpha)\leq 1$, thus $F(\alpha)\leq G(\alpha)$ on $R_1$ when $K\geq 2$. 

Let us now consider $R_2$ and again, behavior of $g$ on level lines of $f$. Since $$\left.g'_\alpha\right|_{\alpha+\beta=C}=-K(\sin\alpha+\sin\beta)< 0$$ we see that $g$ is strictly decreasing on any level line throughout $R_2$ (the only 0 of derivative is achieved at $(\pi,0)$, on the boundary). However let us compare the values of the two functions on two edges of the region: $R_{2}^{bottom}=[\pi/2,\pi] \times 0$ and $R_{2}^{right}= \pi \times [0,\pi/2]$. On the bottom $f$ and $g$ are identical up to a multiple of $K$, so inequality holds. On the right edge, $\left.g\right|_{R_{2}^{right}}=K(1-\cos\beta)$, while $\left.f\right|_{R_{2}^{right}}=1-|\cos(\pi+\beta)|=1-\cos\beta\leq K(1-\cos\beta)$. This amounts to (\ref{fqTi}) holding on right ends of all level line segments in $R_2$, while $g$ stays above and $f$ the same as their respective right end values throughout  the level line segments.

While the cases of $R_3$ and $R_4$ are analogous to $R_2$ and $R_1$ respectively, they are in fact unnecessary to consider as we shall shortly discover. 

We would like now to remove assumption $\gamma=\alpha+\beta$. This is easy when $\alpha,\beta\in R_{1}$. We cover $R_{1}$ with lower left corner region, $R_{1}^{\llcorner}=\{\alpha+\beta\leq\pi/2\}$, and upper right corner one, $R_{1}^{\urcorner}=\overline{R_{1}\setminus R_{1}^{\llcorner}}$. Since $f(\gamma)$ increases on $R_{1}^{\llcorner}$ (while $\gamma \in [0,\pi/2]$), we can use (\ref{tiRG}) and (\ref{fqTi}) to obtain: $f(\gamma)\leq f(\alpha+\beta)\leq g(\alpha,\beta)$. When $\alpha,\beta\in\hspace*{-0.0em}R_{1}^{\hspace*{-0.0em}\urcorner}$ we have 
\begin{equation}
\label{abve1}
\left.g\right|_{R_{1}^{\urcorner}}>1
\end{equation}
according to (\ref{midway}), while $f\leq 1$ always.

Similar to $R_{1}$ above, $R_{2}$ splits into $R_{2}^{\ulcorner}=\{\alpha-\beta\leq\pi/2\}$ and $R_{2}^{\lrcorner}=\overline{R_{2}\setminus R_{2}^{\ulcorner}}$. We use the fact that $g$ is symmetric against the line $\alpha=\pi/2$. Thus $\left.g\right|_{R_{2}^{\ulcorner}}>1$ due to (\ref{abve1}), and the inequality holds on $R_{2}^{\ulcorner}$. For $\alpha,\beta\in R_{2}^{\lrcorner}$ (the mirror image of $R_{1}^{\llcorner}$) we notice that 
\begin{equation}
\label{mirror_chain}
g(\alpha,\beta)=g(\pi-\alpha,\beta)\geq f(\pi-\alpha+\beta)=f(\alpha-\beta),
\end{equation}
with the first equality due to $g$-symmetry against $(\alpha,\beta) \mapsto (\pi-\alpha,\beta)$, the last one due to 
$f$-symmetry against $\gamma \mapsto \pi-\gamma$, and the middle inequality due to (\ref{fqTi}).
However $\alpha-\beta\geq \pi/2$ on $R_{2}^{\lrcorner}$, so together with (\ref{tiRG}) we obtain:
\begin{equation}
\label{range}
\pi/2 \leq \alpha-\beta \leq \gamma \leq \pi,
\end{equation}
as $\gamma$ also represents an angle between two vectors. Couple (\ref{range}) with the fact that $f$ decreases on $[\pi/2,\pi]$, and we can extend (\ref{mirror_chain}) as follows:
\begin{equation}
\label{kill}
f(\alpha-\beta)\geq f(\gamma).
\end{equation}
Combining (\ref{mirror_chain}) with (\ref{kill}) we obtain the desired result on $R_{2}$.

$R_3$ and $R_4$ can be treated via the following trick. Consider $-Y$ instead of $Y$ in Theorem 
\ref{qTLemma}. 
Angle $\gamma$ will not change, while $(\alpha,\beta) \mapsto (A,B)=(\pi-\alpha,\pi-\beta)$. 
Triangle inequality must hold with the new angles too: 
\begin{equation}
\label{tiRG2} 
|A-B|\leq\gamma\leq A+B. 
\end{equation}
Thus for $(\alpha,\beta) \in R_3 \text{ or } R_4$ the corresponding $(A,B)$, in $R_2$ or $R_1$ respectively, will satisfy 
\begin{equation*}
f(\gamma)\leq g(A,B)
\end{equation*}
with the same $\gamma$. Also $g$ is symmetric with respect to maps $x \mapsto \pi-x$ in either argument, hence $g(A,B)=g(\alpha,\beta)$, and we are done with $R_3$, $R_4$.

To prove sharpness, let us consider the case when $\gamma=\alpha+\beta$, and $\alpha=\beta$. Then dividing the left side of (\ref{qTL}) by its right side, and taking a limit as $\alpha \to 0$ we have: $$\lim_{\alpha \to 0}\left(\frac{1-\cos{2\alpha}}{2-2\cos\alpha}\right) \leq K \iff K\geq \lim_{\alpha \to 0}\left(\frac{2-2\cos^2{\alpha}}{2-2\cos\alpha}\right)=2$$ 
Thus $K$ cannot be less than 2 and our estimate is sharp.
\end{proof}

\end{document}